\documentclass[12pt,reqno]{amsart}
\usepackage {amsmath,amsfonts,enumerate, epsfig,
  indentfirst,  graphicx,euscript, amssymb,amsthm,amscd}

\usepackage[dvips]{psfrag}

\usepackage[latin1]{inputenc}

\makeindex

\usepackage{amssymb}
\usepackage[latin1]{inputenc}
\usepackage{ epsfig}
\usepackage{enumerate}
\usepackage{indentfirst}
\usepackage{euscript}
\usepackage{graphicx}
\usepackage{amsmath}
\usepackage{amsfonts}
\usepackage{amssymb}
\makeindex

\makeindex
 \theoremstyle{plain}
\newtheorem{proposition}{Proposition}

\theoremstyle{definition}

\newtheorem{theorem}{Theorem}

  \title[ Dense Principal  Lines on Tori  Embedded  in  $\mathbb R^3$
]{Tori Embedded in $\mathbb R^3$ with Dense   Principal  Lines  }

 \author{ R.   Garcia and J. Sotomayor}


 \thanks{The  authors are fellows of CNPq and
done this work under the project CNPq 473747/2006-5}

\begin{document}

   \begin{abstract} In this paper
  are given
 examples  of
tori  $\mathbb T^2$  embedded  in $\mathbb R^3$
with all  their  principal lines
 dense. These examples are obtained by
 stereographic projection of deformations of the Clifford torus in $\mathbb S^3$.

\end{abstract}

 \maketitle

\section{ Introduction}

Let $\alpha:\mathbb M \to\mathbb S^3$ be
an
immersion
 of class
 $ C^{r }, r\geq 3,$
of a smooth, compact and oriented two-dimensional manifold
$\mathbb M  $ into  the three dimensional sphere ${\mathbb
S}^{3}$, endowed with the canonical inner product $<.,.>$ of
$\mathbb R^4$.

 The {\em Fundamental Forms} of $\alpha $ at
a point $p$ of ${\mathbb M}$ are the symmetric bilinear forms on
${\mathbb T}_p \mathbb M$ defined after  Spivak  \cite{spi},  as
follows:

 $$\aligned I_{\alpha }(p;v,w)=& <D\alpha
(p;v),D\alpha (p;w)>,\\
 II_{\alpha}(p;v,w)=& <-DN_{\alpha}(p;v),D\alpha(p;w)>.\endaligned
$$

 \noindent Here,  $N_{\alpha }$ is the positive unit normal of the immersion $\alpha$
 tangent to
 $\mathbb S^3$ i.e.
$<N_\alpha,\alpha>=0$.

In $\mathbb S^3$,  the coefficients of
the
 second fundamental form
$II_\alpha$ in a local
 positive chart $(u,v)$,  relative to the normal
vector $N=\alpha\wedge \alpha_u\wedge \alpha_v$, are given by:
 $$ \text{\small $ e=\frac{  \det[\alpha, \alpha_u,\alpha_v,\alpha_{uu}]}{
\sqrt{   EG -F^2}}, \;
 f=\frac{  \det[\alpha, \alpha_u,\alpha_v,\alpha_{uv}]}{\sqrt{EG-F^2}},\;
g=\frac{ \det[\alpha,\alpha_u,\alpha_v,\alpha_{vv}]}{\sqrt{ EG-F^2
}}.$}  $$
Here  $E=<\alpha_u, \alpha_u>$, $F=<\alpha_u, \alpha_v>$ and $G=<\alpha_v, \alpha_v>$
are the coefficients of the first fundamental form $I_\alpha$.

The eigenvalues of $II_\alpha- \kappa I_\alpha=0 $ are called  principal
curvatures and the corresponding eigenspaces are called principal directions.

 Outside the set of  umbilic  point ${\mathcal U}_\alpha$
 where the  two principal curvatures
coincide,  the principal directions define
  two line fields, called the {\em principal line fields} of
$\alpha $, which  are of class $C^{r-2}$.
The integral curves of
 the
 principal line fields on $\mathbb M\setminus
{\mathcal U}_{\alpha }$  are called {\it principal curvature
lines}.

The principal
 curvature
   lines can be assembled in two one-dimensional
orthogonal foliations which will be denoted by ${\mathcal
F}_1(\alpha)$ and ${\mathcal F}_2(\alpha)$.
The triple ${\mathcal P}_\alpha=\{ {\mathcal F}_1(\alpha),{\mathcal F}_2(\alpha), {\mathcal U}_\alpha\}$
 is called the {\it principal configuration} of the immersion $\alpha$, \cite{gs1, gs2}.

In a local chart $(u,v)$ the principal  directions of
 an immersion $\alpha$ are defined by the implicit differential equation

\begin{equation}
(Fg-Gf)dv^2+(Eg-Ge)dudv+(Ef-Fe)du^2=0
\end{equation}

A  principal line $\gamma$ is called {\it  recurrent} if
$\gamma\subseteq L(\gamma)$, where $L(\gamma)=\alpha(\gamma)\cup
\omega(\gamma)$ is the {\it limit set} of $\gamma$, and it is
called {\it dense } if ${L(\gamma)}=\mathbb M$.
 Recall that $\alpha$ - limit set, $\alpha(\gamma)$, of an oriented leaf $\gamma$ is
 the set of limit points of  all convergent sequences of the form $\gamma(s_n
 )$,
 where $s_n$ tends to the positive extreme of definition of the
 oriented leaf. The same for the  $\omega$ - limit set,
 $\omega(\gamma)$, exchanging positive by negative in the previous
 definition.
  Clearly, $L(\gamma)$ does not depend on the orientation of the
  leaf.

There is a great similarity between the properties of principal
lines in the cases of surfaces in   Euclidean and in   Spherical
spaces. They correspond to each other by means of a stereographic
projection which, being conformal,  preserves the principal lines.

The study of principal curvature lines on  surfaces $\mathbb M$
immersed in
 $\mathbb R^3$ and $\mathbb S^3$ is a classical
subject of Differential Geometry. See \cite{dc},
\cite{da}, \cite{spi}, \cite{st}.
 
 In the works of Gutierrez  and Sotomayor,
   see \cite{gs1, gsln, gs2},
  ideas coming from the Qualitative Theory of Differential Equations and Dynamical Systems such as
 Structural Stability, Recurrence (and its elimination) and Genericity
 were introduced into the subject of Principal Curvature  Lines.

 Other differential equations of Classical
 Geometry such as asymptotic lines, lines of arithmetic,
 geometric and harmonic mean curvatures,  have been studied in  papers by Garcia, Gutierrez and
 Sotomayor, 1999,
  and
 Garcia and Sotomayor
 1997, 2001, 2002, 2003.
  A unification theory for these differential equations
   was
  achieved  by Garcia and Sotomayor, 2004.
 A recent survey on classical and recent developments on   principal curvature lines, with updated references,
  can be found in \cite{survey}.

 The interest  on  the  study  of foliations with recurrent and dense
 leaves goes back to Poincar\'e, Birkhoff, Denjoy, Peixoto, among others.

Examples of tori embedded  in $\mathbb R^3$
 having only one  of the
two principal foliation with dense leaves have been given in
\cite{glsabc, gsln, gs2}.

In this paper are
 provided
examples of  embedded tori in $\mathbb R^3$
 with both principal
foliations  having dense leaves.
These  examples, established in Theorem \ref{th:2}, are obtained as  the stereographic projections
 of special deformations of the Clifford torus in $\mathbb S^3$. These special
 deformations are defined in the proof of Theorem \ref{th:1p}. See Section \ref{sc:3}.

  \section{Preliminary Calculations}

In this section  will be obtained the variational equations of a
quadratic implicit differential equation. The results
  will be
applied in Section \ref{sc:3}.

\begin{proposition}\label{prop:eqee} Consider a family of  quadratic
 differential equations, depending on a parameter $\epsilon$, of the form

\begin{equation} \aligned \label{eq:qe} a&(u,v,\epsilon) dv^2+ 2b(u,v,\epsilon) dudv+ c(u,v,\epsilon)du^2=0,\\
a&(u,v,0)=c(u,v,0)=0, \;\; b(u,v,0)=1
\endaligned
 \end{equation}

Let $v(u,v_0,\epsilon)$
 be a
solution of \eqref{eq:qe} with $v(u,v_0,0)=v_0$ and
$u(u_0,v,\epsilon)$
 be a
solution of \eqref{eq:qe} with $u(u_0,v,0)=u_0$. Then the
following variational equations hold:

\begin{equation}\label{eq:veep}\aligned
   c_\epsilon \; & + 2  v_{\epsilon u}=0,\;\;\; \;\;   a_\epsilon \;   + 2  u_{\epsilon v}=0,  \\
  c_{\epsilon\epsilon}& +2c_{v\epsilon} v_\epsilon -2 b_\epsilon c_\epsilon  +2  v_{u\epsilon\epsilon}=0\\
  a_{\epsilon\epsilon}& +2a_{u\epsilon} u_\epsilon -2 b_\epsilon a_\epsilon  +2  u_{v\epsilon\epsilon}=0.
\endaligned
\end{equation}

\end{proposition}
\begin{proof}
Differentiation of \eqref{eq:qe} written as
$$a(u,v,\epsilon) (\frac{dv}{du})^2+2 b(u,v,\epsilon) \frac{dv}{du}+c(u,v,\epsilon)=0,\;\;\; v(u,v_0,0)=v_0$$
 leads to:
\begin{equation}\label{eq:ve} (a_\epsilon +a_v v_\epsilon )  (\frac{dv}{du})^2 + 2 a\frac{dv}{du} v_{\epsilon u}+2(b_\epsilon +b_v v_\epsilon )\frac{dv}{du}+ 2b v_{\epsilon u}+c_\epsilon +c_v v_\epsilon =0.\end{equation}
Here $$a_v=\frac {\partial a}{\partial v}, \; a_\epsilon =\frac {\partial a}{\partial \epsilon },\;\;  a_{\epsilon u}= a_{u \epsilon  }=\frac {\partial^2 a}{\partial \epsilon \partial u}=\frac {\partial^2 a}{\partial\epsilon \partial u  }. $$
Analogous notation for    $b=b(u,v(u,v_0,\epsilon),\epsilon)$, $c=c(u,v(u,v_0,\epsilon),\epsilon)$ and   the solution $v(u,v_0,\epsilon)$.

Evaluation of equation \eqref{eq:ve} at $\epsilon =0 $ results in:

$$ c_\epsilon  + 2  v_{\epsilon u}=0.$$

Differentiating   equation  \eqref{eq:ve} and evaluating at $\epsilon=0$  leads to:
\begin{equation}\label{eq:vee}\aligned
 c_{\epsilon\epsilon}&+ 2c_{v\epsilon} v_\epsilon +4b_\epsilon v_{\epsilon u} +2b v_{u\epsilon\epsilon}=0\\
  c_{\epsilon\epsilon}&+2c_{v\epsilon} v_\epsilon -2 b_\epsilon c_\epsilon  +2  v_{u\epsilon\epsilon}=0
\endaligned
\end{equation}

Similar calculation gives the variational equations
 for $u_\epsilon$ and $u_{\epsilon\epsilon}$.
This ends the proof. \end{proof}

\section{ Double Recurrence for Principal Lines}\label{sc:3}

 \begin{theorem} \label{th:1p} There are
 analytic embeddings $\alpha:  \mathbb T^2\to  \mathbb S^3$ such
that all leaves of both principal foliations, ${\mathcal F}_i(\alpha),\; (i=1,2),$   are dense in $\mathbb T^2$.
\end{theorem}

\begin{proof}
Let $N(u,v)=(\alpha \wedge  \alpha_u\wedge \alpha_v)  /|\alpha \wedge  \alpha_u\wedge \alpha_v | (u,v)$ be the normal vector to the Clifford torus  $C=\mathbb S^1(r)\times \mathbb S^1(r)\subset \mathbb S^3$, where $\mathbb S^1(r)=\{(x,y)\in \mathbb R^2:\; x^2+y^2=r^2\}$ and $r=\sqrt{2}/2$, parametrized by

$$\alpha(u,v)=\frac{\sqrt{2}}{2}\big( \cos u, \sin u, \cos v, \sin v\big).$$

We have that,
$$ N_\alpha(u,v)=\frac{\sqrt{2}}2 \big( -\cos u, -\sin u, \cos v, \sin v      \big).   $$
Let $c(u,v)=h(u,v) N_\alpha(u,v)$, $h$ being a $2\pi-$ double periodic function having the symmetry
  property
 $h(u,v)=h(v,u)$.

For $\epsilon\ne 0 $ small consider  the one parameter family of embedded tori defined by:
\begin{equation}\label{eq:cehh}  
              \alpha^{\epsilon}  (u,v)=\frac{ \alpha(u,v)
+\epsilon c(u,v)}{ |\alpha(u,v)+\epsilon c(u,v)|}   .\end{equation}
Let $N_{\alpha^{\epsilon} }
=\alpha^{\epsilon}\wedge(\alpha^{\epsilon})_u\wedge(\alpha^{\epsilon})_v$
be the normal map of $\alpha^{\epsilon}$.

The coefficients of the first and second fundamental forms of
$\alpha^{\epsilon}$ are given by:
$$ \text{ \footnotesize
            $ \aligned E(u,v,\epsilon)=&  \frac{1- \epsilon
 h+2\epsilon^2(h^2+h_u^2)-2\epsilon^3 h^3+\epsilon^4 h^4}{2(1+\epsilon^2 h^2)^2}\\
F(u,v,\epsilon)=&   \frac{ \epsilon^2h_uh_v }{(1+\epsilon^2 h^2)^2} \\
G(u,v,\epsilon)=& \frac{1 + \epsilon h+2\epsilon^2(h^2+h_v^2)+2\epsilon^3 h^3+\epsilon^4 h^4}{2(1+\epsilon^2 h^2)^2} \\
e(u,v,\epsilon)=&\frac{(1+\epsilon h)[1+ \epsilon(2h_{uu}-h) + \epsilon^2(4h_u^2-2hh_{uu}-h^2)+\epsilon^3 h^3]}{4(1+\epsilon^2 h^2)^2}\\
f(u,v,\epsilon)=& \frac{\epsilon[   h_{uv} +\epsilon^2h(h_uh_v-hh_{uv} ) ]}{2 (1+\epsilon^2 h^2)^2} \\
g(u,v,\epsilon)=& \frac{(1-\epsilon h)[-1+ \epsilon(2h_{vv}-h) + \epsilon^2(h^2-4h_v^2+2hh_{vv})+\epsilon^3 h^3]}{4(1+\epsilon^2 h^2)^2}\\
\endaligned $}$$

From the expressions above,
 it is important to record the
following symmetry relations.

\begin{equation}\label{eq:s12}\text{ \footnotesize
            $  \aligned
E(u,v,\epsilon)=&E(v,u, \epsilon), F(u,v,\epsilon)=F(v,u, \epsilon),G(u,v,\epsilon)=G(v,u, \epsilon)\\ e(u,v,\epsilon)=&e(v,u, \epsilon), \;f(u,v,\epsilon)=f(v,u, \epsilon),\;g(u,v,\epsilon)=g(v,u, \epsilon)  \endaligned$}
\end{equation}

 Then the coefficients of the differential equation of principal
 lines  of $\alpha^{\epsilon}$,
after multiplication by  $- 4(1+\epsilon^2 h^2)^4$, are given by:

$$
\text{ \footnotesize
            $ \aligned
 L(u,v,\epsilon)= & (Fg-Gf)(u,v,\epsilon)=  \epsilon h_{uv}+   \epsilon^2  (2hh_{uv}+  h_u h_v) \\
+& (-2h_uh_vh_{vv}+2h_v^2h_{uv}+h_{uv}h^2+2h_uh_v h)\epsilon^3+ (2h_u h_v h^2+ 4 h_u h_v^3)\epsilon^4\\
&+(-h^4 h_{uv}+2h_u h_vh^2 h_{vv} + 4h^3 h_u h_v-2h_v^2h^2 h_{uv})\epsilon^5\\
+&( -2h^5h_{uv} +5h_u h_v h^4)\epsilon^6+(2h^5 h_u h_v-h^6 h_{uv})\epsilon^7\\
M(u,v,\epsilon)=&(Eg-Ge)(u,v,\epsilon)= 1+\epsilon ( h_{uu}-h_{vv})
+(2h (h_{vv}+h_{uu})+3(h_v^2+h_u^2))\epsilon^2 \\
+&   (-h^2h_{vv}+6h_u^2 h-2h_u^2 h_{vv} + 2 h_v^2h_{uu} -6hh_v^2 +h^2h_{uu})\epsilon^3 \\
+&  (-6h_v^2h^2-2h^2+8h_u^2 h_v^2 +6 h_u^2 h^2)\epsilon^4\\
+&  ( h^4 h_{vv} -h^2 h_{uu} -8h^3 h_v^2-2h^2 h_v^2 h_{uu} + 2 h^2 h_u^2 h_{vv}+8 h^3 h_u^2)\epsilon^5\\
+  & (-2h^5h_{vv}-2h^5 h_{uu}
+7 h^4 h_v^2+7 h^4 h_u^2 )\epsilon^6 \\
+&  ( -h^6 h_{uu} + h^6 h_{vv} +2h^5 h_u^2-2h^5 h_v^2)\epsilon^7+  h^8\epsilon^8  \\
N(u,v,\epsilon)=& (Ef-Fe)(u,v,\epsilon)= - \epsilon h_{uv}+   \epsilon^2  (2hh_{uv}+  h_u h_v) \\
+& ( 2h_uh_vh_{uu}-2h_u^2h_{uv}-h_{uv}h_v^2- h_{uv}h^2)\epsilon^3+ (2h_u h_v h^2+ 4 h_u^3 h_v)\epsilon^4\\
&+(h^4 h_{uv}-2h_u h_v h^2 h_{uu}- 4h^3 h_u h_v+2h_u^2h^2 h_{uv})\epsilon^5\\
+&( -2h^5h_{uv} +5h_u h_v h^4)\epsilon^6+(-2h^5 h_u h_v+h^6 h_{uv})\epsilon^7\\
\endaligned $}
$$

By the equations above  and from the symmetry  property  $h(u,v)=h(v,u)$,
 it follows that

\begin{equation}\label{eq:lmnh}  \text{\footnotesize 
            $
L(u,v,\epsilon)=  L(v,u,\epsilon), \;
  M(u,v,\epsilon)=M(v,u,\epsilon),\;
N(u,v, \epsilon)= N(v,u,  \epsilon).
 $}
\end{equation}
At $\epsilon=0$ it follows that

\begin{equation}\label{eq:ez}\text{ \small
            $ \aligned
  L(u,v,0)=&0,\; M(u,v,0)=1, \;N(u,v,0)=0, \;
M_\epsilon(u,v,0)=h_{uu}-h_{vv}\\
N_\epsilon(u,v,0)=&-h_{uv},\;N_{\epsilon v} (u,v,0)=-h_{uvv},
\; N_{\epsilon\epsilon}(u,v,0)=2(2h h_{uv}+h_u h_v)
\endaligned $}
\end{equation}

Consider the reflection $\sigma:\mathbb T^2\to \mathbb T^2$  defined by $\sigma(u,v)=(v,u)$
 and the circle $ \Sigma=\{(u,u): u\in \mathbb R\}$
 transversal to both principal foliations ${\mathcal F}_i(\alpha^{\epsilon}).$

By equation \eqref{eq:lmnh} it follows that $\sigma({\mathcal
F}_1(\alpha^{\epsilon}))={\mathcal F}_2(\alpha^{\epsilon})$ and
$\sigma({\mathcal F}_2(\alpha^{\epsilon}))={\mathcal
F}_1(\alpha^{\epsilon})$, that is the diffeomorphism $ \sigma:
\mathbb T^2\to \mathbb T^2$ preserves the principal configuration
of $\alpha^{\epsilon}$, reversing the order of the foliations.
Therefore it also conjugates the
Poincar\'e first
 return maps $(\pi_{\alpha^{\epsilon}})_i:\Sigma\to
\Sigma$, $(i=1,2)$
 of the foliations
 ${\mathcal F}_i(\alpha^{\epsilon}).$ For the basic properties of
 the rotation number see \cite{ka} and \cite{pa}.

In what follows it will be shown that the rotation number of
$(\pi_{\alpha^{\epsilon}})_i$ changes monotonically with
$\epsilon$, provided it is small.

By proposition \ref{prop:eqee} with $(a=L,\; 2b=M,\; c=N)$ the variational equations of the implicit differential equation

$$N(u,v,\epsilon)+M(u,v,\epsilon)\frac{dv}{du} +L(u,v,\epsilon) (\frac{dv}{du})^2=0,$$
with $N(u,v,0)=L(u,v,0)=0$, $M(u,v,0) =1$ and $v(u,v_0,0)=v_0$
are given by:

$$ N_\epsilon \;   +    v_{\epsilon u}=0,\;\;\;
  N_{\epsilon\epsilon}  +2N_{v\epsilon} v_\epsilon -  2M_\epsilon N_\epsilon  +2  v_{u\epsilon\epsilon}=0.$$

Therefore,

$$\aligned v_\epsilon (u,v_0)=&- \int_0^u h_{uv}(u,v_0) du,\\
v_{u\epsilon\epsilon }  (u,v_0)= &  [\big( \int_0^u h_{uv} du\big) h_{uvv}-  h_{uv} (2h+ h_{uu}-h_{vv})-  h_u h_v](u,v_0)\endaligned$$

Taking $h(u,v)=\sin^2(u+v)$ it results from the equations above
and by  equation \eqref{eq:ez} that:

$$\aligned v_\epsilon(u,v_0,0)=&  \sin(2u+2v_0)-\sin(2v_0)\\
v_{\epsilon\epsilon}(u,v_0,0)=&-\frac 32 u- \sin(2u+2v_0)+\frac 7{8}\sin(4u+4v_0)+ \sin(2u)\\
+& \sin(2v_0)- \sin(4v_0+2u)+\frac 1{8}\sin(4v_0)
\endaligned$$

Therefore it follows that:
$$\aligned   v_{\epsilon}(2\pi,v_0,0)-v_{\epsilon}(0,v_0,0)=&0,\;\; v_{\epsilon\epsilon}(2\pi,v_0,0)-v_{\epsilon\epsilon}(0,v_0,0)= - 3\pi.
\endaligned$$

Consider  the Poincar\'e map $\pi_1^{\epsilon}: \{u=0\}\to
\{u=2\pi\}$ relative to the principal foliation ${\mathcal
F}_1(\alpha^{\epsilon})$, defined by
$\pi_1^{\epsilon}(v_0)=v(2\pi, v_0,\epsilon)$.  Here,   the circle
$\{u=0\}=\{u=2\pi\}$ is a leaf of ${\mathcal F}_2(\alpha)$.

Therefore  $\pi^1_0=Id$ and  $\pi_1^{\epsilon}$  has the following
development:

$$\aligned \pi_1^{\epsilon}(v_0 )=& v_0+ \frac{\epsilon^2} 2 v_{\epsilon\epsilon}(2\pi,v_0,0) +O(\epsilon^3)\\
=& v_0- \frac{3\pi}2  \epsilon^2+O(\epsilon^3)\endaligned.$$

Therefore  the  rotation number of $\pi_1^{\epsilon}$ changes
continuously and monotonically  with $\epsilon$.  See chapter 12
of \cite{ka} or  chapter  5 of \cite{pa}.

 The rotation number of $\pi_1^{\epsilon}$  is
 irrational if
 and only if that of $(\pi_{\alpha^{\epsilon}})_1$ is irrational.

Therefore we can take a small  $\epsilon_0 \ne 0$  such that
 the  rotation numbers of $(\pi^{\alpha_{\epsilon_0}})_i:\Sigma\to \Sigma$, $i= 1,\; 2$
 are both irrational. For $\alpha=\alpha_{\epsilon_0}$ all leaves of both principal
foliations are dense in $\mathbb T^2$.
 This ends the proof.
\end{proof}

\begin{theorem} \label{th:2} There are analytic embeddings $\beta:  \mathbb T^2\to  \mathbb R^3$ such
that all leaves of both principal foliations, ${\mathcal F}_i(\beta),\; (i=1,2)$,  are dense in $\mathbb T^2$. See Figure \ref{fig:tb}.

\begin{figure}[htbp]
\begin{center}
  \hskip .5cm\includegraphics[scale=0.40]{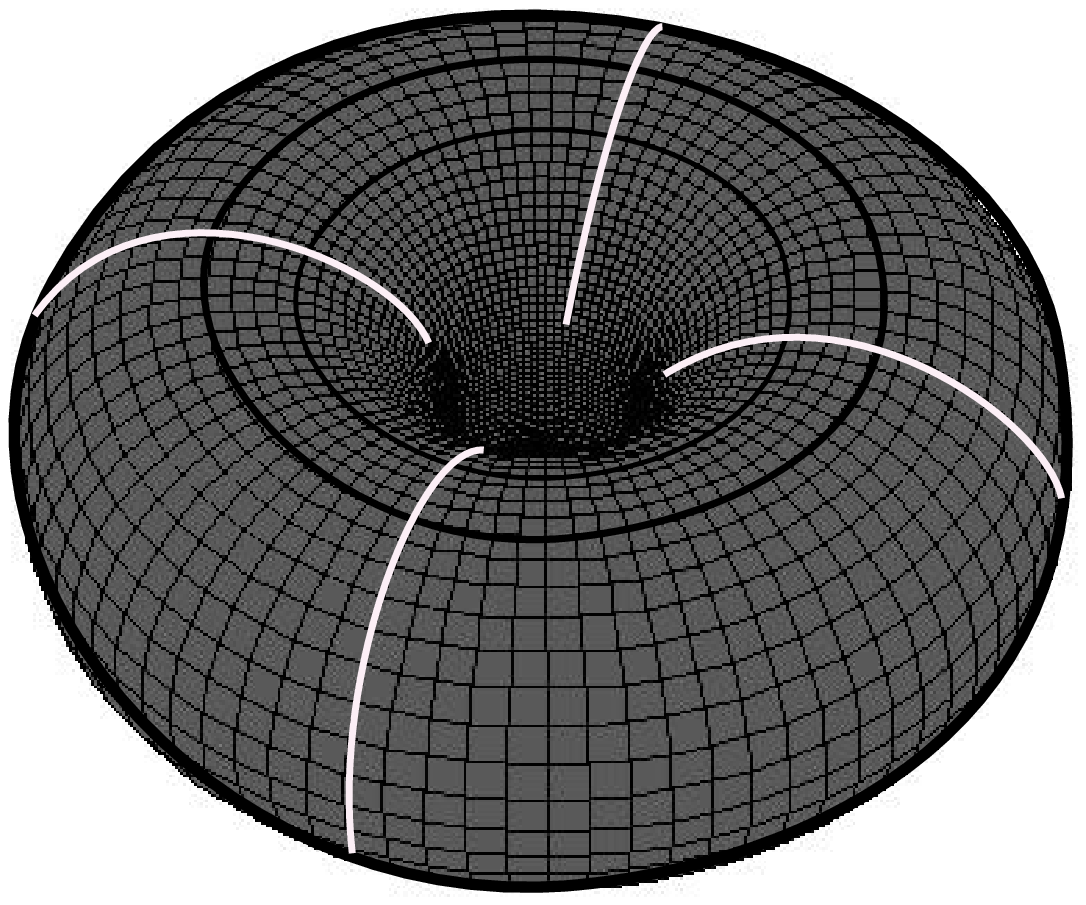} \hskip 2.0cm
  \includegraphics[scale=0.40]{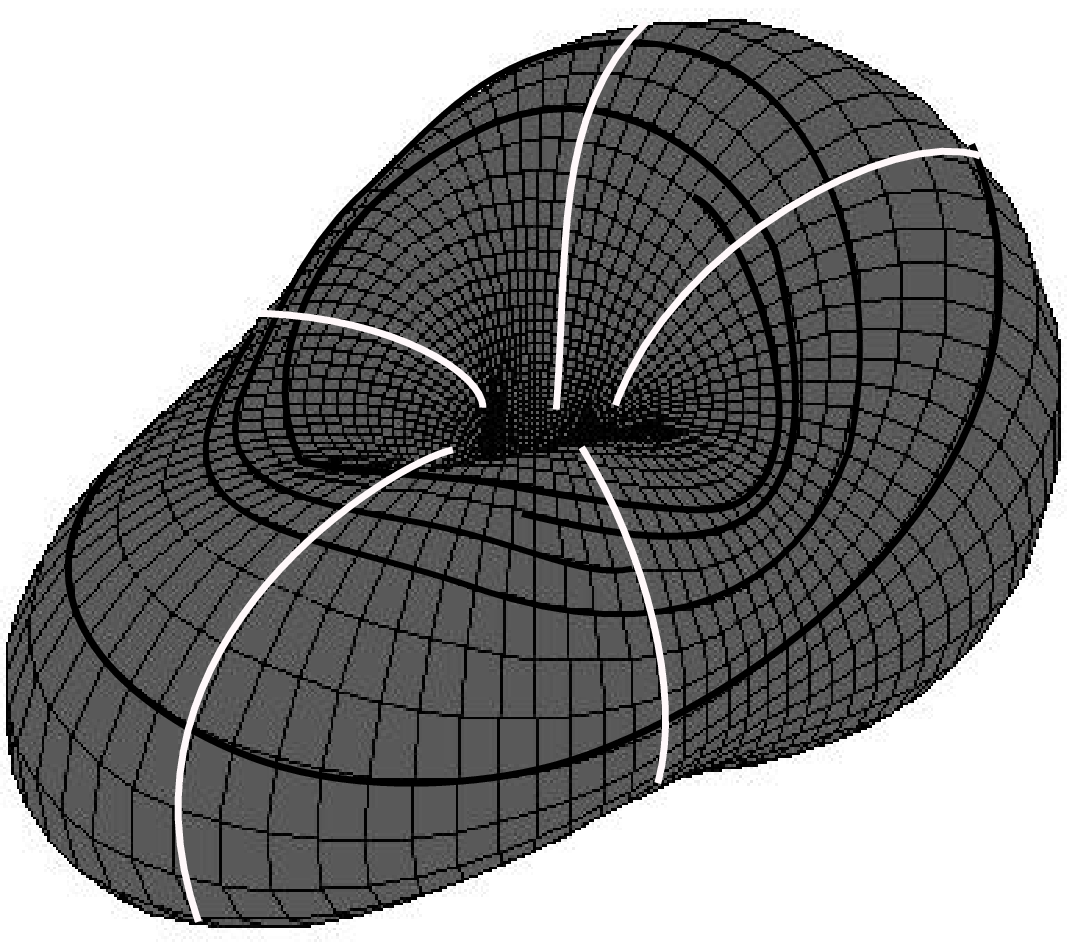}  
\end{center}
\caption{\label{fig:tb}  Stereographic projection of  the Clifford
torus $\alpha^{0}$ (left) and of its $\alpha^{\epsilon} $
deformation for $\epsilon = {\frac 13}$  \, (right)}
\end{figure}

\end{theorem}

\begin{proof} Take the stereographic projection
$\Pi:\mathbb S^3\setminus\{N\}\to \mathbb R^3$, $N\notin
\alpha^{\epsilon}(\mathbb T^2)$, of the example given in Theorem
\ref{th:1p} and define $\beta=\Pi\circ \alpha^{\epsilon}$. As
$\Pi$ is a conformal map the principal lines of  $\beta$ are the
images by $\Pi$ of the principal lines of $\alpha^{\epsilon}$.
\end{proof}

\section{Concluding Comments}

In this paper it  was shown that there exist   embeddings of the
torus in $\mathbb S^3$ and $\mathbb R^3$  with both principal foliations having all
their leaves dense.

The  technique used here  is  based on the second order
perturbation of differential equations.

It is worth mentioning that the consideration of  only the first
order variational equation, see Proposition \ref{prop:eqee}, was
technically  insufficient  to achieve the results of this paper.
 The same can be said for the technique of local bumpy perturbations used to deform
immersions in   \cite{gsln, gs2}, leading   to  examples with one
of the two principal foliations with dense leaves.

\vskip .5cm

\author{\noindent Jorge Sotomayor\\Instituto de Matem\'{a}tica e Estat\'{\i}stica,\\Universidade de S\~{a}o Paulo,
\\Rua do Mat\~{a}o 1010, Cidade Universit\'{a}ria, \\CEP 05508-090, S\~{a}o Paulo, S.P., Brazil \\
\\ Ronaldo Garcia\\Instituto de Matem\'{a}tica e Estat\'{\i}stica,\\
Universidade Federal de Goi\'as,\\CEP 74001-970,
Caixa Postal 131,\\Goi\^ania, GO, Brazil}

\end{document}